\documentclass[12pt]{article}
\usepackage{amsmath}
\usepackage{amsfonts}
\usepackage{amssymb}

\usepackage{amsthm} 
\usepackage{amscd}
\theoremstyle{plain} 
\newtheorem{theorem}{Theorem}[section]
\newtheorem{corollary}[theorem]{Corollary}

\newtheorem{lemma}[theorem]{Lemma}

\newtheorem{definition}[theorem]{Definition}

\numberwithin{equation}{section}

\newcommand{\al}{\alpha}
\newcommand{\be}{\beta}

\newcommand{\vp}{\varphi}

\newcommand{\om}{\omega}
\newcommand{\ot}{\otimes}

\newcommand{\lra}{\longrightarrow}
\newcommand{\lla}{\longleftarrow}

\begin{document}
\title{Hochschild and Simplicial Cohomology}
\author{Jerry Lodder\thanks{Mathematical Sciences; 
Dept. 3MB, Box 30001; New Mexico State University; 
Las Cruces, NM 88003; \texttt{jlodder@nmsu.edu}.}}
\date{}
\maketitle

\noindent
{\bf{Abstract.}} 
We study a naturally occurring $E_{\infty}$-subalgebra of the full
$E_2$-Hochschild cochain complex arising from coherent cochains.  For
group rings and certain category algebras, these cochains detect 
$H^*(B {\cal{C}})$, the simplicial cohomology of the classifying space of the 
underlying group or category, $\cal{C}$.  In this setting the
simplicial cup product of cochains on $B{\cal{C}}$ agrees with the
Gerstenhaber product and Steenrod's cup-one product of cochains agrees
with the pre-Lie product.  We extend the idea of coherent cochains to
algebras more general than category algebras and dub the resulting
cochains autopoietic.  Coefficients are from a commutative ring $k$ 
with unit, not necessarily a field.



\section{Introduction}
In this paper we extend Gerstenhaber's idea of coherent Hochschild
cochains \cite{Gerstenhaber2017} to algebras $A$ with a free
$k$-module basis
$$  {\cal{B}} = \{ \varphi_{\om} \ | \ \om \in \Omega \},  $$
where $\Omega$ is an arbitrary index set, and the product on $A$ is
induced from a product among the $\vp$s with 
$$  \vp_{\al} \vp_{\be} = \mu_{\al , \be} \, \vp_m,  $$
where $\vp_m \in {\cal{B}}$, $\mu_{\al, \be} \in k$.  Here $k$ is a
commutative ring with unit, not necessarily a field, and $A$ is unital
as well.   We call the resulting cochains $f \in {\rm{Hom}}_k(A^{\ot
  n}, \, A)$ autopoietic (self-generating), and they satisfy the
condition
$$  f(\vp_1 \ot \vp_2 \ot \, \ldots \, \ot \vp_n ) =
\lambda_{\vp_1, \, \ldots \, , \vp_n} (\vp_1 \cdot \vp_2 \cdot \ldots \cdot
\vp_n), $$ 
where $\lambda_{\vp_1, \, \ldots \, , \vp_n}$ is an element of $k$
depending on $(\vp_1, \, \vp_2, \, \ldots \, , \vp_n )$. For group
rings, poset algebras and certain category algebras, the autopoietic
cochains detect $H^*( B {\cal{C}})$, the simplicial (singular)
cohomology of the classifying space of the underlying group or
category, $\cal{C}$, expressed
as a direct summand of  $HH^*(k [{\cal{C}}]; \, k [{\cal{C}}])$, the
Hochschild cohomology of the category algebra $k [{\cal{C}}]$.  When
the elements $\mu_{\al , \be} \in k$ are idempotents (as for group rings
or category algebras), the autopoietic cochains form an
$E_{\infty}$-subalgebra of the full $E_2$-Hochschild cochain complex
with the Gerstenhaber product agreeing with the simplicial cup product
of cochains on $B {\cal{C}}$ and the pre-Lie product agreeing with
Steenrod's cup-one product of cochains on $B {\cal{C}}$.  We believe
the result for $E_{\infty}$-algebras is new and the autopoietic
cochains provide a common and direct setting for separate calculations
appearing in \cite{Lodder1, Lodder2}.      

Hochschild cohomology, $HH^*$, is a topic of current research as a
target for topological quantum field theories 
\cite{Costello, Kontsevich2008} with the Gerstenhaber product in $HH^*$ 
corresponding to the corbordism between
two disjoint copies of the unit circle, $S^1$, and one copy of $S^1$.
The product on $HH^*$ is known to be graded commutative via two
possible cochain homotopies, namely the pre-Lie products 
\cite{Gerstenhaber1962}.  These two homotopies are themselves not
cochain homotopic, forming the structure of an $E_2$-algebra on
Hochschild cochains.  The action of $E_2$-operads on the Hochschild
cochain complex that are themselves homotopic to the little disk
operad, i.e., a solution to the Deligne conjecture, has been a topic
research studied in 
\cite{Kontsevich1999, Kontsevich2000, McClure-Smith, Tamarkin, Voronov}. 
In this paper, however, we study a subcomplex of Hochschild cochains
that carries an $E_{\infty}$-algebra structure, namely the subcomplex
of autopoietic cochains.  For a group ring, $k[G]$, the Frobenius
algebra structure on $k[G]$, Equation \eqref{Frobenius}, can be used to identify
Hochschild cochains that are dual to the constant loops on $BG$, which
yields Lemma \eqref{AP-iso}.  The construction detecting the constant
loops on $BG$ generalizes to algebras that do not necessarily carry
the structure of a Frobenius algebra, such as poset algebras and
category algebras.  For
a poset with a finite number of objects, 
the autopoietic cochains on the poset algebra $k[{\cal{C}}]$ compute 
$$  HH^*( k [ {\cal{C}}]; \, k [ {\cal{C}}]) \simeq H^*(B{\cal{C}}), $$
where $\cal{C}$ is the category associated to a poset (see \S 2).

The idempotent condition, $\mu^2 = \mu$, is needed to map cochains on
the bar construction to Hochschild cochains. It is also needed to
compare Steenrod's cup-one product with Gerstenhaber's pre-Lie
product.  The subcomplex of autopoietic Hochschild cochains, however, 
can be defined without the idempotent condition, and is of independent
interest.  
Section 2 of the paper contains the specifics for group
rings and poset algebras.  Section 3 provides the general definition
of autopoietic cochains, and establishes when they form an
$E_{\infty}$-subalgebra of the full Hochschild cochain complex.  
Section 4 contains an application of autopoietic cochains to compute 
$HH^*(k [ {\cal{C}}]; \, k [ {\cal{C}}])$, where $\cal{C}$ is a
certain type of endomorphism-isomorphism category formed by an amalgam
of groups and a poset, which extends results in \cite{Xu}, where only
field coefficients are used.  Our results also apply to infinite,
discrete groups.  Finally we identify a topological space whose
simplicial cohomology is isomorphic to 
$HH^*(k [ {\cal{C}}]; \, k [ {\cal{C}}])$, where $\cal{C}$ is an
amalgam of groups and a poset.

\section{Group Rings and Poset Algebras}

In this section we discuss a common setting for both group rings and
poset algebras for constructing certain Hochschild cochains that
detect $H^*( B{\cal{C}})$, the simplicial cohomology of the
classifying space of the underlying group or category.  Let $G$ be a
discrete group, finite or infinite, considered as a category $\cal{C}$
with one object and morphisms given by the elements of $G$.  The
classifying space $BG$ can be identified with $B{\cal{C}}$ and the
group ring $k[G]$ can be identified with the category algebra $k
[{\cal{C}}]$.  Here $k$ is a commutative ring with unit $1 \in k$.   

The Hochschild homology $HH_* (k[G]; \, k[G])$ can be computed from
the complex
$$  k[G] \overset{b}{\lla} k[G]^{\ot 2} \overset{b}{\lla}
\ldots \overset{b}{\lla} k[G]^{\ot n} \overset{b}{\lla} k[G]^{\ot
  (n+1)} \overset{b}{\lla} \ldots \, ,$$
where for $z = g_0 \ot g_1 \ot \ldots \ot g_n \in k[G]^{\ot (n+1)}$,
\begin{align*}
b(z) = \; & 
 \sum_{i=0}^{n-1} (-1)^i \, g_0 \ot g_1 \ot \ldots \ot g_i g_{i+1}
\ot \ldots \ot g_n  \\
& + (-1)^{n} g_n g_0 \ot g_1 \ot g_2 \ot \ldots \ot g_{n-1} .
\end{align*}
This is the homology of a simplicial set known as the cyclic bar
construction \cite[7.3.10]{Loday}, denoted in this paper as
$N_*^{\rm{cy}}(G)$, with $N_n^{\rm{cy}}(G) = G^{n+1}$, $n \geq 0$.  
It is known \cite[7.4.5]{Loday} that $H_*(BG)$ is a direct summand of 
$HH_* (k[G]; \, k[G])$.  Moreover, $H_*(BG; \, k)$ can be computed
from the subcomplex $k[N_*^{\rm{cy}}(G, \, e)]$ of Hochschild
chains, where
$$  N_*^{\rm{cy}}(G, \, e) = \{ (g_0, \, g_1, \, \ldots \, ,  g_n) \in
G^{n+1} \ | \ g_0g_1 \ldots g_n = e \}.  $$
Here $e$ is the identity element of $G$.  Let $B_*(G)$ be the
simplicial bar construction on $G$, where $B_n(G) = G^n$, $n \geq 0$, 
with the standard face and degeneracy maps \cite[B.12, Example 1]{Loday}.  
There is a well-known chain isomorphism 
$$  k[B_*(G)] \lra k[N_*^{\rm{cy}}(G, \, e)]  $$
induced by the map of simplicial sets
\begin{align*}
& \iota : B_n(G) \to N_n^{\rm{cy}}(G), \ \ \ \iota : G^n \to
G^{n+1}, \\
& \iota(g_1, \, g_2, \, \ldots \, , g_n) = \big( (g_1 g_2 \ldots
g_n)^{-1}, \, g_1, \, g_2, \, \ldots \, , g_n \big).  
\end{align*}
The inverse chain map is induced by another map of simplicial sets
\begin{align*}
& \pi : N_n^{cy}(G) \to B_n(G), \ \ \ \pi: G^{n+1} \to G^n, \\
& \pi (g_0, \, g_1, \, g_2, \, \ldots \, , g_n) = (g_1, \, g_2, \,
\ldots \, , g_n).
\end{align*}
Let $(HH^*( k[G]), \ b^*)$ denote the cohomology of the
${\rm{Hom}}_k$-dual of the $b$-complex.  Specifically, $(HH^*( k[G]),
\ b^*)$ is the cohomology of the following with $n \geq 1$: 
$$  \ldots  \overset{b^*}{\lra} {\rm{Hom}}_k (k[G]^{\ot n}, \, k) 
\overset{b^*}{\lra} {\rm{Hom}}_k (k[G]^{\ot (n+1)}, \, k)
\overset{b^*}{\lra} \ldots \ . $$
Clearly, $H^*(BG; \, k)$ is a direct summand of 
 $(HH^*( k[G]), \ b^*)$, $G$ finite or infinite.  

Let $({\rm{Hom}}_k( k[G]^{\ot *}, \, k[G]), \ \delta)$ be the standard
complex \eqref{H-cochain-complex}
for computing Hochschild cohomology, $HH^*( k[G]; \, k[G])$.  Define
an inner product 
\begin{equation} \label{Frobenius}
\langle \ , \; \rangle : k[G] \ot k[G] \to k  
\end{equation}
given on group elements $g$, $h \in G$ by
$$  \langle g, \, h \rangle = \begin{cases} 1, & h = g^{-1} \\
                                            0, & h \neq g^{-1}.
                              \end{cases}  $$
Then extend $\langle \ , \; \rangle$ to be linear on the tensor product
$k[G] \ot k[G]$.  
There is a cochain map \cite{Lodder1}
$$  \Phi_n : ({\rm{Hom}}_k (k[G]^{\ot n} , \, \, k[G]),\ \delta) \to ({\rm{Hom}}_k 
 (k[G]^{\ot (n+1)} , \, \, k),\ b^*), \ \ \   n \geq 0, $$
given by 
$$  \Phi_n (f) (g_0 \ot g_1 \ot \ldots \ot g_n) =
\langle g_0, \ f(g_1 \ot g_2 \ot \ldots \ot g_n) \rangle ,  $$
where $f : k[G]^{\ot n} \to k[G]$ is a $k$-linear map and each $g_i
\in G$.  Let $\Phi = \{ \Phi_n \}_{n \geq 0}$.  For $G$ finite, $\Phi$
is an isomorphism of cochain complexes.  For $G$ infinite, $\Phi$ is
only injective on cochains.  Nonetheless, 
$HH^*( k[G]; \, k[G])$ contains a copy of $H^* (BG; \, k)$, $G$ finite
or infinite \cite[2.11.2]{Benson}, and $HH^*( k[G]; \, k[G])$ splits as the direct
sum of $H^*(BG; \, k)$ and a complementary submodule.  This achieved
in this paper 
by writing ${\rm{Hom}}_k(k[G]^{\ot *}, \ k[G])$ as the direct sum of
two cochain complexes, one being the autopoietic cochains and the
other the non-autopoietic cochains.  

\begin{definition}
For $n \geq 1$, $f \in {\rm{Hom}}_k(k[G]^{\ot n}, \ k[G])$ is
autopoietic if for all $(g_1, \, g_2, \, \ldots \, , \, g_n) \in G^n$,
we have
$$  f(g_1 \ot g_2 \ot \ldots \ot g_n) = \lambda_{g_1, \ldots , g_n}
(g_1 g_2 \ldots g_n)  $$
for some $\lambda_{g_1, \ldots , g_n} \in k$, depending on 
$(g_1, \, g_2, \, \ldots \, , \, g_n)$.  
\end{definition}
\begin{definition}
For $n \geq 0$, $f \in {\rm{Hom}}_k(k[G]^{\ot n}, \ k[G])$ is strictly
autopoietic if $f$ is autopoietic and additionally for $n = 0$ and 
$f \in {\rm{Hom}}_k(k, \ k[G])$, we have $f(1) = \lambda e$ for some
$\lambda \in k$.
\end{definition}
Let $AP(k[G]^{\ot n})$ denote the $k$-submodule of 
${\rm{Hom}}_k(k[G]^{\ot n}, \ k[G])$ consisting of (strictly)
autopoietic cochains.  It is shown in the next section that
$$  AP( k[G]^{\ot *} ) = \sum_{n \geq 0} AP(k[G]^{\ot n})  $$
is a subcomplex of $\big( {\rm{Hom}}_k(k[G]^{\ot *}, \ k[G]), \
\delta \big)$.  
\begin{lemma}  \label{AP-iso}
There is a natural isomorphism
$$  H^* ( AP( k[G]^{\ot *} ) ) \simeq H^*(BG; \, k),  $$
where $G$ is a finite or an infinite group.
\end{lemma}
\begin{proof}
For $n \geq 0$, we show that
$$  AP(k[G]^{\ot n}) \simeq {\rm{Hom}}_k ( k[N_n^{\rm{cy}} (G, \, e)],
\ k).  $$
First, let $f \in AP(k[G]^{\ot n})$ and $(g_0, \, g_1, \, \ldots  ,
g_n) \in N_n^{\rm{cy}} (G, \, e)$.  Then $g_0^{-1} = (g_1 g_2 \ldots
g_n)$
and
\begin{align*}
\Phi (f) (g_0 \ot g_1 \ot \ldots \ot g_n) &  = \langle g_0, \ f(g_1
\ot \ldots \ot g_n) \rangle \\
& = \langle g_0, \ \lambda_{g_1, \ldots , g_n} (g_1 g_2 \ldots g_n)
\rangle \\
& = \lambda_{g_1, \ldots , g_n}.
\end{align*}
For $h_0$, $h_1$, $\ldots$, $h_n \in G$ with $h_0^{-1} \neq (h_1 h_2
\ldots h_3)$ and $f$ autopoietic, $\Phi (f) (h_0 \ot h_1 \ot \ldots
\ot h_n) = 0$.  Thus, 
$$  \Phi : AP(k[G]^{\ot n}) \to {\rm{Hom}}_k ( k[N_n^{\rm{cy}} (G, \,
e)], \ k).  $$ 

Conversely, given $\al \in {\rm{Hom}}_k ( k[N_n^{\rm{cy}} (G, \,
e)], \ k)$, then $\al (g_0 \ot g_1 \ot \ldots \ot g_n) \in k$ for all
$(g_0, \, g_1, \, \ldots \, , g_n) \in G^{n+1}$ with
$g_0^{-1} = (g_1 g_2 \ldots g_n)$.  Thus $\al$ depends only on 
$(g_1, \, g_2, \, \ldots \, , g_n)$.  Set
$$  \al (g_0 \ot g_1 \ot \ldots \ot g_n) = \lambda_{g_1, \ldots ,
  g_n}.  $$
Define $\Psi (\al) \in AP (k[G]^{\ot n})$ by
$$  \Psi (\al) (g_1 \ot g_2 \ot \ldots \ot g_n) = 
\lambda_{g_1, \ldots , g_n} (g_1 g_2 \ldots g_n).  $$
It follows quickly that
$$ \Psi : {\rm{Hom}}_k ( k[N_n^{\rm{cy}} (G, \, e)], \ k)  \to
AP (k[G]^{\ot *})  $$
is a cochain map with $\Phi \circ \Psi = {\bf{1}}$ and
$\Psi \circ \Phi = {\bf{1}}$.  In the next section we prove that in a
more general setting $\Psi$ is a cochain map.  Since the cohomology of 
${\rm{Hom}}_k ( k[N_n^{\rm{cy}} (G, \, e)], \ k)$ can be canonically
identified with $H^*(BG; \, k)$, the lemma follows.
\end{proof}

Geometrically, the autopoietic cochains are dual to the constant loops
on $BG$, which can be seen from the geometric realization
$$  | N^{cy}_*(G) | \simeq {\rm{Maps}}(S^1, \, BG).  $$
The space ${\rm{Maps}}(S^1, \, BG)$ is commonly referred to as the free
loop space on $BG$.  
The non-autopoietic cochains are defined to mirror the quotient 
$$ {\rm{Hom}}_k(k[G]^{\ot n}, \ k[G])/AP(k[G]^{\ot n}).  $$ 
For $g \in G$, define projection maps
$$  \pi_g : k[G] \to k  $$
as follows.  Let $z = \sum_{h \in G} c_h h \in k[G]$.  For $G$
infinite, only finitely many of $c_h$s are non-zero.  Let $\pi_g
(z) = c_g$, i.e., the coefficient of $g$.    
\begin{definition}
A cochain $f \in {\rm{Hom}}_k ( k[G]^{\ot n}, \ k[G])$ is
non-autopoietic if for all $(g_1, \, g_2, \, \ldots  , g_n) \in G^n$,
we have
$$  \pi_h (f(g_1 \ot g_2 \ot \ldots \ot g_n)) = 0, \ \ \
h = g_1 g_2 \ldots g_n . $$
For $n = 0$, $f \in {\rm{Hom}}_k(k, \, k[G])$ is non-autopoietic if 
for $  f(1)  =  \sum_{g \in G} c_g \, g$, 
we have $c_e = 0$, i.e., the coefficient on $e$ is zero.
\end{definition}
Let $NP(k[G]^{\ot n})$ denote the submodule of non-autopoietic
cochains in ${\rm{Hom}}_k( k[G]^{\ot n}, \ k[G])$.
\begin{lemma}
The submodule $NP(k[G]^{\ot *} ) = \sum_{n \geq 0}NP(k[G]^{\ot n})$ 
forms a subcomplex of 
$$ \big( {\rm{Hom}}_k(k[G]^{\ot *}, \ k[G]), \ \delta \big). $$
\end{lemma}
\begin{proof}
Let $\vp \in NP(k[G]^{\ot n})$.  Then
\begin{align*}
& (\delta \vp )  (g_1 \ot g_2 \ot \ldots \ot g_{n+1}) = 
  g_1 \vp (g_2 \ot \ldots \ot g_{n+1}) + \\
&  \Big( \sum_{i = 1}^n (-1)^i \vp (g_1 \ot \ldots \ot g_i g_{i+1} \ot
\ldots \ot g_{n+1}) \Big)
 + (-1)^{n+1} \vp (g_1 \ot \ldots \ot g_n ) g_{n+1}.
\end{align*}
If $g_1 \vp (g_2 \ot \ldots \ot g_{n+1})$ contains a non-zero summand
of the from $\lambda (g_1 g_2 \ldots g_{n+1})$, then
$$  \vp (g_2 \ot \ldots \ot g_{n+1}) = g_1^{-1} g_1 \vp (g_2 \ot
\ldots \ot g_{n+1}) $$
would contain a non-zero summand of the form 
$\lambda (g_2 g_3 \ldots g_{n+1})$, contradicting that 
$\vp \in  NP(k[G]^{\ot n})$.  Clearly
$$  \pi_h \vp (g_1 \ot \ldots \ot g_i g_{i+1} \ot
\ldots \ot g_{n+1}) = 0, \ \ \ h = (g_1 g_2 \ldots g_{n+1}).  $$
Also, by multiplication on the right by $g_{n+1}^{-1}$, we see that
$\vp (g_1 \ot \ldots \ot g_n ) g_{n+1}$ cannot contain a non-zero
summand of the form  $\lambda (g_1 g_2 \ldots g_{n+1})$.  Thus,
$\delta \vp \in NP(k[G]^{\ot (n+1)})$.  
\end{proof}
\begin{lemma}  \label{direct-sum}
There is a direct sum splitting of cochain complexes
$$  {\rm{Hom}}_k(k[G]^{\ot *}, \, k[G]) \simeq AP(k[G]^{\ot *}) \oplus
NP(k[G]^{\ot *}),  $$
where in dimension zero, cochains in $AP^*$ are considered to be strictly
autopoietic.  
\end{lemma}
\begin{proof}
Let $f \in {\rm{Hom}}_k(k[G]^{\ot n}, \, k[G])$.  Define 
\begin{align*}
&  \pi : {\rm{Hom}}_k(k[G]^{\ot n}, \, k[G]) \to AP(k[G]^{\ot n}) \\
& \pi (f)(g_1 \ot g_2 \ot \ldots \ot g_n) = \pi_h \big( f(g_1 \ot
\ldots \ot g_n) \big) (g_1 g_2 \ldots g_n), \\
& h = g_1 g_2 \ldots g_n \ {\rm{for}} \ n > 0, \ \ \ h = e  \
{\rm{for}} \ n = 0.
\end{align*} 
Let $f_1 = \pi (f)$ and $f_2 = f - f_1$.  Then $f_2 \in NP(k[G]^{\ot
  n})$ and $ f \mapsto f_1 \oplus f_2$ establishes the splitting
$$  {\rm{Hom}}_k(k[G]^{\ot *}, \, k[G]) \simeq AP(k[G]^{\ot *}) \oplus
NP(k[G]^{\ot *}).  $$
\end{proof}

In the next section we study $AP(k[G]^{\ot *})$ as an
$E_{\infty}$-subalgebra of 
$$ {\rm{Hom}}_k(k[G]^{\ot *}, \, k[G]) .  $$  
In general, however, $NP(k[G]^{\ot *})$ is not closed under the 
Gerstenhaber product.  We note in passing that the cochain map
$$  \Phi :  {\rm{Hom}}_k (k[G]^{\ot n} , \,  k[G]) \to ({\rm{Hom}}_k 
 (k[G]^{\ot (n+1)} , \, \, k)  $$
can be used to induce a different product structure on    
 $NP(k[G]^{\ot *})$, namely the simplicial cup product inherited from the
 cup product on the cochain complex
${\rm{Hom}}_k(k[G]^{\ot + 1}, \, k)$ as described in \cite{Lodder1}.
The simplicial cup product exploits the simplicial structure from the
cyclic bar construction.  In fact,  $NP(k[G]^{\ot *})$ inherits an action of
the cup-$i$ products via $\Phi$ \cite{Lodder1}.

We briefly describe a class of algebras whose Hochschild cohomology is
completely determined by autopoietic cochains, namely poset algebras.  
Let ${P} = \{i, \, j, \, \ldots \, \}$ be a finite poset of
cardinality $N$ containing no cycles and partial order denoted
$\preccurlyeq$.  Let $k$ be a commutative ring with unit $1 \in k$,
and let $A$ be the poset algebra of upper triangular matrices
with $k$-module basis given by $e_{ij}$, $i \preccurlyeq j$, subject to the
relations
$$  e_{ij} e_{k \ell} =  e_{i \ell}, \  j = k,  \ \ \
            e_{ij} e_{k \ell} = 0, \  j \neq k.   $$  
Note that $A$ can be considered as a category algebra $k[ {\cal{C}}]$,
where $\cal{C}$ is the category with objects given by the elements of
$P$ and morphisms given by
$$  {\rm{Mor}}(i, \ j) = \begin{cases} e_{ij}  & i \preccurlyeq j \\
                                       \emptyset  & {\rm{otherwise.}}
\end{cases}  $$  
Composition of morphisms (from left to right) agrees with the product
above.  

Note that $k[\cal{C}]$ contains a separable subalgebra
$$  E = \langle e_{11}, \ e_{22}, \ \ldots \, , \ e_{NN} \rangle ,$$
namely the $k$-module generated by the elements $e_{ii}$, $i = 1$, 2,
$\ldots \,$, $N$.  It is well known \cite{Gerstenhaber1983,
Gerstenhaber2014} that $HH^* (k [{\cal{C}}]; \, k[{\cal{C}}])$ can
be computed from a subcomplex of ${\rm{Hom}}_k( k[{\cal{C}}]^{\ot n}, \,
k[{\cal{C}}])$ given by $E$-relative cochains.  Specifically, a cochain 
$f \in {\rm{Hom}}_k( k[{\cal{C}}]^{\ot n}, \, k[{\cal{C}}])$, $n \geq
1$, is
$E$-relative if for composable (from left to right) morphisms
$e_{i_0 i_1}$, $e_{i_1 i_2}$, $\ldots\, $, $e_{i_{n-1} i_n}$,
$i_0 \preccurlyeq i_1 \preccurlyeq i_2 \preccurlyeq \ldots
\preccurlyeq i_n$, we have
$$  f(e_{i_0 i_1} \ot e_{i_1 i_2} \ot \ldots \ot e_{i_{n-1} i_n} =
(\lambda_{i_0, i_1, \ldots \, , i_n}) e_{i_0 i_n},   $$
where $\lambda_{i_0, i_1, \ldots \, , i_n} \in k$.  Moreover, for
morphisms $m_1$, $m_2$, $\ldots \,$, $m_n$ if $m_i$ is not composable
with $m_{i+1}$, then
$$  f( m_1 \ot m_2 \ot \ldots \ot m_n) = 0.  $$
Note that in $k[ {\cal{C}}]$, 
$e_{i_0 i_1} \cdot e_{i_1 i_2} \cdot \ldots \cdot e_{i_{n-1} i_n} = 
 e_{i_0 i_n}$
and $m_1 \cdot m_2 \cdot \ldots \cdot m_n = 0$ ($m_i$ not composable
with $m_{i+1}$).  Thus, an $E$-relative
cochain $f : k[ {\cal{C}} ] ^{\ot n} \to k[ {\cal{C}} ]$ satisfies the
condition to be autopoietic, Definition \eqref{autop}, although not strictly
autopoietic \eqref{st-autop}, since for $n = 0$, the complex ${\rm{Hom}}_k (k,
\, k [ {\cal{C}}])$ is replaced with
$$  {\rm{Hom}}_k ( k, \, k[{\rm{Obj}}({\cal{C}})] ),  $$
where ${\rm{Obj}}({\cal{C}})$ denotes the objects of the category
$\cal{C}$.

\section{Autopoietic Cochains}

Let $A$ be an associative algebra with unit $e \in A$ over a
commutative ring $k$ with unit $1 \in k$.  Recall that the Hochschild
cohomology of $A$ with coefficients in $A$, $HH^*(A; \, A)$, can be
computed from the standard $A \ot A^{\rm{op}}$ resolution of the
multiplication map $\mu : A \ot A \to A$ \cite[IX.6]{Cartan-Eilenberg}.
The resulting cochain complex is:  
\begin{align} \label{H-cochain-complex} 
& {\rm{Hom}}_k(k, \, A) \overset{\delta}{\lra} {\rm{Hom}}_k(A, \, A)
\overset{\delta}{\lra} \ldots \\
& \ldots  \overset{\delta}{\lra} 
{\rm{Hom}}_k(A^{\ot n}, \, A) \overset{\delta}{\lra} 
{\rm{Hom}}_k(A^{\ot (n+1)}, \, A) \overset{\delta}{\lra}  \ldots \, ,
\end{align}
where, for a $k$-linear map $f: A^{\ot n} \to A$, $\delta f : A^{\ot
  (n+1)} \to A$ is given by
\begin{align*}
& (\delta f)(a_1, \, a_2, \, \ldots \, , a_{n+1}) = a_1 f(a_2, \, \ldots
\, , a_{n+1})  \, + \\
& \Big ( \sum_{i=1}^n (-1)^i f(a_1, \, a_2, \, \ldots \, , a_i a_{i+1}, \,
\ldots \, , a_n) \Big) + (-1)^{n+1} f(a_1, \, a_2, \, \ldots \, , a_n)
a_{n+1}.  
\end{align*}
For the special case of $n = 0$ and $f \in {\rm{Hom}}_k (k, \, A)$, 
$(\delta f)(a_1) = a_1 f(1) - f(1)a_1$.

Now suppose that $A$ is a free $k$-module with $k$-module basis given
by
$$  {\cal{B}} = \{ \varphi_{\om} \ | \ \om \in \Omega \},  $$
where $\Omega$ is an arbitrary index set.  Furthermore, suppose that
the product in $A$ is induced by a product among the $\vp_{\om}$s with
$$  \vp_{\al} \vp_{\be} = \mu_{\al, \be} \, \vp_m,  $$
where $\vp_m \in {\cal{B}}$, $\mu_{\al, \be} \in k$, and $\mu_{\al,
  \be}$ is an idempotent, i.e., $\mu_{\al, \be}^2 = \mu_{\al, \be}$.
For some $\al$ and $\be$, it could happen that $\mu_{\al, \be} = 0$.  

\begin{definition} \label{autop}
For $n \geq 1$, a cochain $f \in {\rm{Hom}}_k(A^{\ot n}, \, A)$ is
autopoietic if for all $\vp_i \in {\cal{B}}$, we have
$$  f(\vp_1 \ot \vp_2 \ot \ldots \ot \vp_n) =
\lambda_{\vp_1, \ldots, \vp_n} (\vp_1 \vp_2 \ldots \vp_n),  $$
where $\lambda_{\vp_1, \ldots, \vp_n}$ is
an element of $k$ depending on $(\vp_1, \, \vp_2, \, \ldots \, , \,
\vp_n)$.
\end{definition}
\begin{definition} \label{st-autop}
For $n \geq 0$, a cochain $f \in {\rm{Hom}}_k(A^{\ot n}, \, A)$ is
strictly autopoietic if $f$ is autopoietic and for $f \in
{\rm{Hom}}_k(k, \, A)$, we have $f(1) = \lambda e$ for some $\lambda
\in k$.
\end{definition}

\begin{lemma}
The autopoietic cochains (and strictly autopoietic cochains) form a
subcomplex of ${\rm{Hom}}_k(A^{\ot *}, \, A)$.
\end{lemma}
\begin{proof}
Let $f \in {\rm{Hom}}_k(A^{\ot n}, \, A)$ be (strictly) autopoietic.
For $\vp_i$, $\vp_{i+1} \in {\cal{B}}$, let
$\vp_i \vp_{i+1} = \mu_i \cdot \vp_{m_i}$.  Then 
\begin{align*}
& (\delta f) (\vp_1 \ot \vp_2 \ot \ldots \ot \vp_{n+1}) =
\vp_1 \cdot f(\vp_2 \ot \ldots \ot \vp_{n+1})  \\
& + \Big( \sum_{i=1}^n (-1)^i f(\vp_1 \ot \ldots \ot \vp_i \vp_{i+1} \ot
\ldots \ot \vp_{n+1}) \Big) \\ 
& + (-1)^{n+1} f (\vp_1 \ot \ldots \ot \vp_n)
\cdot \vp_{n+1}  \\
& = \lambda_{\vp_2, \ldots , \vp_{n+1}} (\vp_1 \vp_2 \ldots \vp_{n+1}) \\
& + \sum_{i=1}^n (-1)^i \lambda_{\vp_1, \ldots, \vp_{m_i}, \ldots ,
  \vp_{n+1}} (\vp_1 \vp_2 \ldots \vp_i \vp_{i+1} \ldots \vp_{n+1}) \\
& + (-1)^{n+1} \lambda_{\vp_1, \ldots , \vp_n} (\vp_1 \ldots \vp_n
\vp_{n+1}) \\
& = \lambda_{\vp_1 , \vp_2, \ldots , \vp_{n+1}} (\vp_1 \vp_2 \ldots \vp_{n+1})
\end{align*}
for some $\lambda_{\vp_1 , \vp_2, \ldots , \vp_{n+1}} \in k$.  
\end{proof}

\begin{lemma}
The complex of (strictly) autopoietic cochains is closed under the
Gerstenhaber product and the pre-Lie product.
\end{lemma}
\begin{proof}
Let $f \in {\rm{Hom}}_k(A^{\ot p}, \, A)$ and  
$g \in {\rm{Hom}}_k(A^{\ot q}, \, A)$ be (strictly) autopoietic.  Then
the Gerstenhaber product $f \underset{G}{\cdot} g \in
{\rm{Hom}}_k(A^{\ot (p+q)}, \, A)$ is determined by
\begin{align*}
& (f \underset{G}{\cdot} g) (\vp_1 \ot \ldots \ot \vp_p \ot \vp_{p+1} \ot
\ldots \ot \vp_{p+q}) \\
& = f(\vp_1 \ot \ldots \ot \vp_p) \cdot g(\vp_{p+1} \ot \ldots \ot
\vp_{p+q}) \\
& = (\lambda_{\vp_1 , \ldots , \vp_n}) 
(\lambda_{\vp_{p+1}, \ldots , \vp_{p+q}}) 
(\vp_1 \ldots \vp_p)(\vp_{p+1} \ldots \vp_{p+q}) \\
& = (\lambda_{\vp_1, \ldots , \vp_{p+q}})(\vp_1 \vp_2 \ldots
\vp_{p+q})
\end{align*}
for some $\lambda_{\vp_1, \ldots , \vp_{p+q}} \in k$.  Thus, 
$f \underset{G}{\cdot} g$ is (strictly) autopoietic.

For $j = 0$, 1, 2, $\ldots \,$, $p-1$, the $j$th partial composition 
$$  f \underset{(j)}{\circ} g \in {\rm{Hom}}_k(A^{\ot (p+q-1)}, \, A)  $$ 
is determined by
\begin{align*}
& f \underset{(j)}{\circ} g (\vp_1 \ot \vp_2 \ot \ldots \ot
\vp_{p+q-1}) \\
& = f(\vp_1 \ot \ldots \ot \vp_j \ot g(\vp_{j+1} \ot \ldots \ot \vp_{j+q})
\ot \vp_{j+q+1} \ot \ldots \vp_{p+q-1}) \\
& = (\lambda_{\vp_{j+1}, \ldots , \vp_{j+q}}) 
f(\vp_1 \ot \ldots \ot \vp_j \ot (\vp_{j+1} \ldots \vp_{j+q}) \ot
\vp_{j+q+1} \ot \ldots \ot \vp_{p+q-1}) \\
& = \lambda_{\vp_1, \ldots , \vp_{p+q-1}} (\vp_1 \vp_2 \ldots \vp_{p+q-1}) 
\end{align*} 
for some $\lambda_{\vp_1, \ldots , \vp_{p+q-1}} \in k$.  Since each
partial composition $f \underset{(j)}{\circ} g$ is autopoietic, it
follows that the pre Lie-product
\begin{equation}  \label{pre-Lie}
f \circ g = \sum_{j=0}^{p-1}
(-1)^{(p-1-j)(q-1)} f \underset{(j)}{\circ} g 
\end{equation}
is autopoietic (with any sign convention).
\end{proof}

Thus, the autopoietic cochains are closed under the Gerstenhaber
product and the partial compositions of the endomorphism operad as
well.  Using McClure and Smith's \cite{McClure-Smith} description of the
sequence operad, it follows quickly that the autopoietic cochains are closed
under the $E_2$-operad given by sequences of complexity less than or
equal to two.  We, however, study an $E_{\infty}$-structure on
autopoietic cochains arising from Steenrod's cup-$i$ products.

Let $\{ A^{\ot n} \}$ be the bar construction on $A$ with
face maps $d_i : A^{\ot n} \to A^{\ot (n-1)}$, $n \geq 2$, $i = 0$, 1,
$\ldots \,$, $n$, given by
\begin{equation} \label{bar-construction}
d_i(a_1 \ot a_2 \ot \ldots \ot a_n) = \begin{cases}
a_2 \ot a_3 \ot \ldots \ot a_n, & i = 0, \\
a_1 \ot \ldots \ot a_i a_{i+1} \ot \ldots \ot a_n, &  1 \leq i \leq 
 n-1, \\
a_1 \ot a_2 \ot \ldots \ot a_{n-1}, & i = n.
\end{cases}
\end{equation}
At present we do not require that $d_0$, $d_1: A \to k$ be given, but
these will be included when they exist (such as for group rings or
category algebras).  Let $d: A^{\ot n} \to A^{\ot (n-1)}$ be the
boundary operator $d = \sum_{i=0}^n (-1)^i d_i$.  The coboundary map
$$  d^* : {\rm{Hom}}_k(A^{\ot (n-1)}, \, k) \to
{\rm{Hom}}_k(A^{\ot n}, \, k)  $$
is given by the Hom-$k$ dual of $d$.  

Consider $A$ as a free
$k$-module with basis ${\cal{B}} = \{ \varphi_{\om} \ | \ \om \in \Omega \}$
as above.  Let $AP(A^{\ot *}, \, A)$ denote the subcomplex of
(strictly) autopoietic Hochschild cochains.  There is $k$-module map
$$  \Psi : {\rm{Hom}}_k (A^{\ot n}, \, k) \to
AP(A^{\ot n}, \, A)  $$
induced by
$$  \Psi ( \al ) (\vp_1 \ot \vp_2 \ot \ldots \ot \vp_n) =
\al (\vp_1 \ot \ldots \ot \vp_n) \cdot (\vp_1 \vp_2 \ldots \vp_n),  $$
where $\al \in {\rm{Hom}}_k (A^{\ot n}, \, k)$.  There is also a
$k$-module map
$$  \Phi : AP(A^{\ot n}, \, A) \to {\rm{Hom}}_k (A^{\ot n}, \, k)  $$
induced by
$$  \Phi (f)(\vp_1 \ot \vp_2 \ot \ldots \ot \vp_n) = \begin{cases}
\lambda_{\vp_1, \ldots , \vp_n}, & \vp_1 \vp_2 \ldots \vp_n \neq 0 \\
0, & \vp_1 \vp_2 \ldots \vp_n = 0, \end{cases}  $$
where $f \in AP(A^{\ot n}, \, A)$ and 
$f (\vp_1 \ot \vp_2 \ot \ldots \ot \vp_n) = \lambda_{\vp_1, \ldots ,
\vp_n} (\vp_1 \vp_2 \ldots \vp_n)$.  When $A$ is a group ring or a
polynomial algebra, $\Psi$ and $\Phi$ are $k$-module isomorphisms. 
\begin{lemma}
For $n \geq 1$, we have a commutative diagram
$$ \CD
AP(A^{\ot n}, \, A)  @>{\delta}>>  AP(A^{\ot (n+1)}, \, A)  \\
@A{\Psi}AA   @AA{\Psi}A  \\
{\rm{Hom}}_k( A^{\ot n}, \, k) @>{d^*}>> {\rm{Hom}}_k(A^{\ot (n+1)},\,
k).
\endCD  $$
\end{lemma}
\begin{proof}
Let $\al \in {\rm{Hom}}_k(A^{\ot n}, \, k)$.  For ease of notation,
set $\vp_i \vp_{i+1} = \mu_i \, \vp_{m_i}$, where $\mu_i \in k$ and
$\mu_i^2 = \mu_i$.  Then
\begin{align*}
& (\delta \circ \Psi )(\al)(\vp_1 \ot \vp_2 \ot \ldots \ot \vp_{n+1})\\
& = \vp_1 \, \Psi (\al) (\vp_2 \ot \ldots \ot \vp_{n+1}) \\
& + \sum_{i=1}^n (-1)^i \Psi (\al)(\vp_1 \ot \ldots \ot \vp_i
\vp_{i+1} \ot \ldots \ot \vp_{n+1})  \\
& + (-1)^{n+1} \Psi (\al)(\vp_1 \ot \ldots \ot \vp_n) \vp_{n+1} \\
& = \al(\vp_2 \ot \ldots \ot \vp_{n+1}) \vp_1 \vp_2 \ldots \vp_{n+1} \\
& + \sum_{i=1}^n (-1)^i \mu_i \al( \vp_1 \ot \ldots \ot \vp_{m_i} \ot
\ldots \vp_{n+1}) \, \vp_1 \vp_2 \ldots \vp_{m_i} \ldots \vp_{n+1}  \\
& + (-1)^{n+1} \al (\vp_1 \ot \ldots \ot \vp_n) \vp_1 \vp_2 \ldots
\vp_{n+1} \\
& = (C_1 + C_2 + C_3) \, \vp_1 \vp_2 \ldots \vp_{i} \vp_{i+1} \ldots \vp_{n+1},
\ \ \ {\rm{where}} \\
& \ \ \ \ \ C_1 = \al (\vp_2 \ot \ldots \ot \vp_{n+1}), \ \ \ 
C_3 = (-1)^{n+1} \al (\vp_1 \ot \ldots \ot \vp_n),  \\
& \ \ \ \ \ C_2 = \sum_{i=1}^n (-1)^{i}  \al(\vp_1 \ot \ldots \ot \vp_i \vp_{i+1}
\ot \ldots \ot \vp_{n+1}).  
\end{align*}
On the other hand,
\begin{align*}
& (\Psi \circ d^* ) (\al)(\vp_1 \ot \vp_2 \ot \ldots \ot \vp_{n+1}) \\
& = \Psi \big( d^*(\al) \big) (\vp_1 \ot \vp_2 \ot \ldots \ot
\vp_{n+1}) \\
& = d^*(\al) (\vp_1 \ot \vp_2 \ot \ldots \ot \vp_{n+1}) \, \vp_1 \vp_2
\ldots \vp_{n+1} \\
& = (C_1 + C_2 + C_3) \, \vp_1 \vp_2 \ldots \vp_{n+1}.  
\end{align*}
\end{proof}
Note that for strictly autopoietic cochains and $\{ A^{\ot n} \}_{n
  \geq 0}$ a simplicial $k$-module with $d_0 = d_1 : A \to k$, the
following commutes
$$  \CD
AP(k, \, A)  @>{\delta}>>  AP(A, \, A)  \\
@A{\Psi}AA   @AA{\Psi}A  \\
{\rm{Hom}}_k(k , \, k) @>{d^*}>> {\rm{Hom}}_k(A, \, k).
\endCD  $$
This follows since $d^* = 0$ and $\delta = 0$ in the above special
case.

Using the simplicial structure of $A^{\ot *}$, the simplicial cup
product can be defined on cochains ${\rm{Hom}}_k(A^{\ot *}, \, k)$.
Let $\al \in {\rm{Hom}}_k(A^{\ot p}, \, k)$, 
$\be \in {\rm{Hom}}_k(A^{\ot q}, \, k)$, and
$$  \sigma = \vp_1 \ot \vp_2 \ot \ldots \ot \vp_{p+q} \in A^{\ot (p+q)}.  $$
Then $\al \underset{S}{\cdot} \be \in {\rm{Hom}}_k(A^{\ot (p+q)}, \,
k)$ is given by
$$  (\al \underset{S}{\cdot} \be) (\sigma) =
\al (\big( f_p (\sigma) \big) \, \be \big( b_q (\sigma) \big),  $$
where $f_p (\sigma)$ denotes the front $p$-face of $\sigma$ and $b_q
(\sigma)$ denotes the back $q$-face of $\sigma$.  Specifically,
$$   (\al \underset{S}{\cdot} \be) (\sigma) =
\al (\vp_1 \ot \ldots \ot \vp_p) \be (\vp_{p+1} \ot \ldots \ot \vp_{p+q})
\in k.  $$
\begin{lemma}
The cochain map
$ \Psi : {\rm{Hom}}_k (A^{\ot *}, \, k) \to AP(A^{\ot *} , \, A)  $
is a map of (differential graded) algebras, i.e.,
$$  \Psi (\al \underset{S}{\cdot} \be) =
\Psi (\al) \underset{G}{\cdot} \Psi (\be).  $$
\end{lemma}
\begin{proof} 
We have
\begin{align*}
& \Psi (\al \underset{S}{\cdot} \be) (\sigma) =
(\al \underset{S}{\cdot} \be) (\sigma) \, \vp_1 \vp_2 \ldots \vp_{p+q} \\
& = \al (\vp_1 \ot \ldots \ot \vp_p) \be (\vp_{p+1} \ot \ldots \ot \vp_{p+q})
\, \vp_1 \vp_2 \ldots \vp_{p+q}.
\end{align*}
On the other hand,
\begin{align*}
& \Psi (\al) \underset{G}{\cdot} \Psi (\be) (\sigma)  \\
& = \Psi (\al)(\vp_1 \ot \ldots \ot \vp_p) \, \Psi (\be) (\vp_{p+1} \ot
\ldots \ot \vp_{p+q}) \\
& = \al (\vp_1 \ot \ldots \ot \vp_p) \, \be (\vp_{p+1} \ot \ldots \ot \vp_{p+q})
\, \vp_1 \vp_2 \ldots \vp_{p+q}.
\end{align*} 
\end{proof}

Recall that ${\rm{Hom}}_k( A^{\ot *}, \, k)$ is a graded homotopy
commutative algebra with the chain homotopy between 
$\al \underset{S}{\cdot} \be$ and $\be \underset{S}{\cdot} \al$ given
in term of Steenrod's cup-one product \cite{Mosher}.  Of course, 
${\rm{Hom}}_k( A^{\ot *}, \, k)$ is an $E_{\infty}$-algebra with the
higher homotopies given by the cup-$i$ products \cite{Mosher}.  We
choose the following sign convention for the cup-one product
$$  \al \underset{1, \, S}{\cdot} \be \in 
{\rm{Hom}}_k(A^{\ot (p+q-1)}, \, k).  $$
Let $\sigma = \vp_1 \ot \vp_2 \ot \ldots \ot \vp_{p+q-1} \in A^{\ot
  (p+q-1)}$.  Set
\begin{align*}
& (\al \underset{1, \, S}{\cdot} \be)_j (\sigma)  = A_j \, B_j  ,\\
& A_j = \al ( (d_{j+1}  d_{j+2} \,  \ldots  \, d_{j+q-1}) (\sigma) ), \\
& B_j = \be( (d_0  d_1 \,  \ldots \,  d_{j-1} \,
d_{j+q+1}  d_{j+q+2}  \, \ldots \,  d_{p+q-1}) (\sigma) ) , \\
& (\al \underset{1, \, S}{\cdot} \be)(\sigma)  =
\sum_{j=0}^{p-1} (-1)^{(p-1-j)(q-1)} \, 
(\al \underset{1, \, S}{\cdot} \be)_j (\sigma) .
\end{align*}
\begin{theorem}
The cochain map
$ \Psi : {\rm{Hom}}_k(A^{\ot *}, \, k) \to AP(A^{\ot *}, \, A)  $
is a map of graded homotopy commutative algebras, i.e.,
$  \Psi (\al \underset{1, \, S}{\cdot} \be) = 
\Psi (\al) \circ \Psi (\be)$.
\end{theorem}
\begin{proof}
We show that for $j = 0$, 1, 2, $\ldots \,$, $p-1$, 
$ \Psi  (\al \underset{1, \, S}{\cdot} \be)_j (\sigma)  =
\Psi (\al) \underset{(j)}{\circ} \Psi (\be)$.  
We have
\begin{align*}
& \Psi ( (\al \underset{1, \, S}{\cdot} \be)_j) (\sigma) \\
& = (\al \underset{1, \, S}{\cdot} \be)_j (\sigma) \, \vp_1 \vp_2
\ldots \vp_{p+q-1}  \\
& = A_j B_j \,  \vp_1 \vp_2 \ldots \vp_{p+q-1}, \ \ {\rm{where}} \\
& A_j = \al \big( \vp_1 \ot \vp_2 \ldots \ot (\vp_{j+1}\vp_{j+2} \ldots
\vp_{j+q}) \ot \vp_{j+q+1} \ot \ldots \ot \vp_{p+q-1} \big) \in k , \\
& B_j = \be (\vp_{j+1} \ot \vp_{j+2} \ot \ldots \ot \vp_{j+q}) \in k .
\end{align*}
On the other hand,
\begin{align*}
& \Psi (\al) \underset{(j)}{\circ} \Psi (\be) (\sigma) \\ 
& = \Psi (\al) (\vp_1 \ot \vp_2 \ot \ldots \ot C \ot \vp_{j+q+1} \ot
\ldots \ot \vp_{p+q-1}), \\
& \ \ \ \ C = \Psi (\be) (\vp_{j+1} \ot \vp_{j+2} \ot \ldots \ot \vp_{j+q}) 
= B_j \, \vp_{j+1} \vp_{j+2} \ldots \vp_{j+q}.
\end{align*}
Let $\vp_{j+1} \vp_{j+2} \ldots \vp_{j+q} = \mu \, \vp_m$, where $\mu \in
k$.  Then $\mu^2 = \mu$ and
\begin{align*}
& \Psi (\al) \underset{(j)}{\circ} \Psi (\be) (\sigma) \\
& = B_j \mu \, \Psi (\vp_1 \ot \vp_2 \ot \ldots \ot \vp_m \ot \ldots \ot
\vp_{p+q-1}) \\
& = B_j \mu^2 \, \al (\vp_1 \ot \vp_2 \ot \ldots \ot \vp_m \ot \ldots \ot
\vp_{p+q-1}) \cdot \vp_1  \vp_2 \ldots  \vp_m  \ldots 
\vp_{p+q-1} \\
& = A_j B_j \, \vp_1 \vp_2 \ldots \vp_{j+1} \vp_{j+2} \ldots \vp_{j+q} 
\ldots  \vp_{p+q-1} \\
& = \Psi ( (\al \underset{1, \, S}{\cdot} \be)_j) (\sigma) .
\end{align*}
\end{proof} 
\begin{lemma}
For $n \geq 1$, we have a commutative diagram
$$ \CD
AP(A^{\ot n}, \, A)  @>{\delta}>>  AP(A^{\ot (n+1)}, \, A)  \\
@V{\Phi}VV   @VV{\Phi}V  \\
{\rm{Hom}}_k( A^{\ot n}, \, k) @>{d^*}>> {\rm{Hom}}_k(A^{\ot (n+1)},\,
k).
\endCD  $$
\end{lemma}
\begin{proof}
Let $ f \in AP(A^{\ot n}, \, A)$ and set $\vp_i \vp_{i+1} = \mu_i
\vp_{m_i}$.  Then $\mu_i^2 = \mu_i$.  Thus,
\begin{align*}
& (\delta f) (\vp_1 \ot \vp_2 \ot \ldots \ot \vp_{n+1}) =
\lambda_{\vp_2, \ldots , \vp_{n+1}} (\vp_1 \vp_2 \ldots \vp_{n+1}) \\
& + \sum_{i=1}^n (-1)^i \mu_i \lambda_{\vp_1, \vp_2, \ldots , \vp_{m_i}
  , \ldots , \vp_{n+1}} (\vp_1 \vp_2 \ldots \vp_{m_i} \ldots
\vp_{n+1}) \\
& + (-1)^{n+1} \lambda_{\vp_1, \vp_2, \ldots , \vp_n} (\vp_1 \vp_2
\ldots \vp_{n+1}) \\
& = \big( \lambda_{\vp_2, \ldots , \vp_{n+1}} + (-1)^{n+1} \lambda_{\vp_1,
  \vp_2, \ldots , \vp_n} \big) (\vp_1 \vp_2 \ldots \vp_{n+1}) \\
& + \sum_{i=1}^n (-1)^i \mu_i \lambda_{\vp_1, \vp_2, \ldots , \vp_{m_i}
  , \ldots , \vp_{n+1}} \mu_i (\vp_1 \vp_2 \ldots \vp_{m_i} \ldots
\vp_{n+1}) \\
& = (C_1 + C_2 + C_3) (\vp_1 \vp_2 \ldots \vp_i \vp_{i+1} \ldots
\vp_{n+1}), \ \ \ {\rm{where}} \\
& C_1 = \lambda_{\vp_2, \ldots , \vp_{n+1}}, \ \ \ 
C_2 = \sum_{i=1}^n (-1)^i \mu_i \lambda_{\vp_1, \vp_2, \ldots , \vp_{m_i}
  , \ldots , \vp_{n+1}},  \\
& C_3 = (-1)^{n+1} \lambda_{\vp_1, \vp_2, \ldots , \vp_n}. \\
\end{align*}
Thus, $\Phi (\delta f)(\vp_1 \ot \vp_2 \ot \ldots \ot \vp_{n+1}) = C_1
+ C_2 + C_3$.  On the other hand,
\begin{align*}
& d^*( \Phi (f))( \vp_1 \ot \vp_2 \ot \ldots \ot \vp_{n+1}) \\
& = \Phi (f)(d( \vp_1 \ot \vp_2 \ot \ldots \ot \vp_{n+1})) = C_1 + C_2
+ C_3 .
\end{align*}
\end{proof}

Given $\al \in {\rm{Hom}}_k(A^{\ot p}, \, k)$ and 
$\be \in {\rm{Hom}}_k(A^{\ot q}, \, k)$, the (simplicial) cup-$i$
product 
$\al \underset{i, \, S}{\cdot} \be \in {\rm{Hom}}_k(A^{\ot (p+q-i)},
\, k)$, $i \geq 0$, \cite{Mosher, Steenrod}.  We choose the following sign
convention:  
\begin{align*}
& d^*( \al \underset{i, \, S}{\cdot} \be) = d^*( \al ) \underset{i, \, S}{\cdot} \be 
+ (-1)^{p-1} \al  \underset{i, \, S}{\cdot} d^*( \be ) \\
& \hspace{.9in} + (-1)^p [ (-1)^{(i-1)(p+q+1)}\al  \underset{i-1, \, S}{\cdot} \be - 
(-1)^{pq} \be  \underset{i-1, \, S}{\cdot} \al ].
\end{align*}
An action of the cup-$i$ products can be defined on autopoietic
cochains via pull-back from ${\rm{Hom}}_k(A^{\ot *}, \, k)$.  
\begin{definition}
Let $f \in AP(A^{\ot p}, \, A)$ and $g \in AP(A^{\ot q}, \, A)$.  Then 
$f \underset{i, \, S}{\cdot} g \in AP(A^{\ot (p+q-i)}, \, A)$ is
defined by
$$  f \underset{i, \, S}{\cdot} g = \Psi \big( \Phi (f) 
\underset{i, \, S}{\cdot} \Phi (g) \big).  $$
\end{definition}
Since $\Phi$ and $\Psi$ are maps of cochain complexes, it follows that 
\begin{align*}
& \delta ( f \underset{i, \, S}{\cdot} g) = \delta (f) \underset{i, \, S}{\cdot} g 
+ (-1)^{p-1} f \underset{i, \, S}{\cdot} \delta (g) \\
& \hspace{.9in} + (-1)^p [ (-1)^{(i-1)(p+q+1)}f \underset{i-1, \, S}{\cdot} g - 
(-1)^{pq} g \underset{i-1, \, S}{\cdot} f ].
\end{align*}
Thus, the autopoietic cochains $AP(A^{\ot *} , \, A)$ inherit an
$E_{\infty}$-algebra structure via the cup-$i$ products.

\section{Endomrophism-Isomorphism 
Categories}

Let $P = \{i, \, j, \, \ldots \, \}$ be a finite poset of
cardinality $N$ containing no cycles and partial order denoted
$\preccurlyeq$ as in \S 2. Consider the category $\cal{C}$ with
objects given by the elements of $P$ and two types of morphisms:
\begin{align*}
& {\rm{Mor}} \, (i, \, j) = \begin{cases} e_{ij}, & i \preccurlyeq j,
  \ \ i \neq j , \\
\emptyset , & i \npreceq j , \end{cases} \\
& {\rm{Mor}} \, (i, \, i) = G_i ,
\end{align*}
where $G_i$ is a discrete group, trivial, finite or infinite.  The
identity of $G_i$ is denoted $e_{ii}$.  Composition of morphisms is
from left to right.  Thus,
$$  e_{ij} \, e_{kl} = \begin{cases} e_{il}, & j = k , \\
\emptyset , & j \neq k .  \end{cases}  $$
For $g \in G_i$, $g e_{ij} = e_{ij}$, $i \neq j$.  For $h \in  G_j$,
$ e_{ij} h = e_{ij}$, $i \neq j$.  We call the category $\cal{C}$ an
amalgam of groups and a poset, which is a special case of an
endomorphism-isomorphism category \cite{Xu}, although the set of
morphisms is not necessarily finite.  In any event, the morphisms of
$\cal{C}$ form a free $k$-module basis for the category algebra
$k[{\cal{C}}]$.  
\begin{lemma}  \label{unique-morphism}
Let $i$, $j \in {\rm{Obj}} ( {\cal{C}} )$ with $i \preccurlyeq j$ and
$i \neq j$.  Then there is a unique morphism $e_{ij} \in {\rm{Mor}}(i,
\, j)$ with domain $i$ and codomain $j$.
\end{lemma}
\begin{proof}
Let $i = i_0 \precneqq i_1 \precneqq i_2 \precneqq \, \ldots \,
\precneqq i_k = j$ in the partial order on the objects of $\cal{C}$.  
The symbol $ i \precneqq j$
denotes that $i \preccurlyeq j$ and $i \neq j$.  Let $g_i \in G_i$.
Then 
\begin{align*}
& g_{i_0} \, e_{i_0 i_1} \, g_{i_1} \, e_{i_1 i_2} \, g_{i_2} \, 
\ldots \, g_{i_{k-1}} \, e_{i_{k-1} i_k} \, g_{i_k}  \\
& = e_{i_0 i_1} \, e_{i_1 i_2} \, \ldots \, e_{i_{k-1} i_k} \\
& = e_{ij}.
\end{align*}
\end{proof}
The $k$-module $E = \langle e_{11}, \, e_{22}, \, \ldots , \, e_{NN}
\rangle$ is a separable subalgebra of the category algebra
$k[{\cal{C}}]$.  Thus, a $k[{\cal{C}}] \ot
k[{\cal{C}}]^{\rm{op}}$-projective, acyclic resolution of the
multiplication map
$ k[{\cal{C}}] \underset{k}{\ot} k[{\cal{C}}] \overset{b'}{\lra} k[{\cal{C}}]$,
$b'(a_0 \ot a_1) = a_0 a_1$ is given by
\begin{align*}
& k[{\cal{C}}] \overset{b'}{\lla} k[{\cal{C}}] \underset{E}{\ot} k[{\cal{C}}]
\overset{b'}{\lla} \ldots  \overset{b'}{\lla} k[{\cal{C}}]^{\ot_E  \, n}
\overset{b'}{\lla} k[{\cal{C}}]^{\ot_E (n+1)}
\overset{b'}{\lla} \\
& b'( a_0 \ot a_1 \ot \ldots a_n) = \sum_{i=0}^{n-1} (-1)^{i} 
a_0 \ot \ldots \ot a_i a_{i+1} \ot \ldots \ot a_n . 
\end{align*}
The reader can check that a contracting chain homotopy
$\vp :  k[{\cal{C}}]^{\ot_E (n+1)} \to  k[{\cal{C}}]^{\ot_E (n+2)}$ of
the $b'$-complex is given by
$$  \vp (a_0 \ot a_1 \ot \ldots \ot a_n) = e \ot a_0 \ot a_1 \ot
\ldots \ot a_n ,  $$
where $e = \sum_{i=1}^N e_{ii}$ is the identity element of
$k[{\cal{C}}]$.  

The Hochschild cohomology groups $HH^*(k[{\cal{C}}]; \, k[{\cal{C}}])$
can be computed from the cochain complex
\begin{equation} \label{C-*-complex}
 C^0 \overset{\delta}{\lra} C^1 \overset{\delta}{\lra} C^2 
\overset{\delta}{\lra} \ldots \overset{\delta}{\lra} 
C^n \overset{\delta}{\lra} C^{n+1} \overset{\delta}{\lra} \ldots \, , 
\end{equation}
where $C^n = {\rm{Hom}}_{k[{\cal{C}}] \ot k[{\cal{C}}]^{\rm{op}}}
( k[{\cal{C}}]^{\ot_E (n+2)}, \, k[{\cal{C}}])$.  It follows quickly
that as $k$-modules
$$  C^0 \simeq \sum_{i=1}^N {\rm{Hom}}_k (k_i, \, k[G_i]),  $$
where $k_i = k[e_{ii}]$ is the free $k$-module on one element
$e_{ii}$, which is isomorphic to the free $k$-module on the object $i$
in $\cal{C}$.  For $n \geq 1$, an element $f \in C^n$ is determined by
a $k$-linear map $f : k[{\cal{C}}]^{\ot_E n} \to k[{\cal{C}}]$ with
the property that
$$  f( \al_{i_1 i_2} \ot \al_{i_2 i_3} \ot \ldots \ot \al_{i_n
  i_{n+1}}) \in k[{\rm{Mor}}(i_1, \, i_{n+1})] ,  $$
where $i_1\preccurlyeq  i_2 \preccurlyeq  i_3 \preccurlyeq \ldots 
\preccurlyeq i_{n+1}$ and $\al_{ij} \in {\rm{Mor}}(i, \, j)$.  
For $n \geq 1$ and $i_1 \neq i_{n+1}$, such a cochain $f \in C^n$ is
necessarily autopoietic by lemma \eqref{unique-morphism}.  Also for $n
\geq 1$, the coboundary $\delta f : C^n \to C^{n+1}$ supports a
formulation as
\begin{align*}
& (\delta f) ( \al_{i_1 i_2} \ot \al_{i_2 i_3} \ot \ldots \ot 
\al_{i_{n+1} i_{n+2}} )  = \al_{i_1 i_2} f( \al_{i_2 i_3} \ot \ldots
\ot \al_{i_{n+1} i_{n+2}} ) \\
& + \sum_{j=1}^n (-1)^j f( \al_{i_1 i_2} \ot \ldots \ot
\al_{i_j i_{j+1}} \al_{i_{j+1} i_{j+2}} \ot \ldots \ot 
\al_{i_{n+1} i_{n+2}} ) \\
& + (-1)^{n+1} f( \al_{i_1 i_2} \ot \ldots \ot \al_{i_n i_{n+1}} ) 
\al_{i_{n+1} i_{n+2}} 
\end{align*}
For $n = 0$, $(\delta f)(\al_{i_1 i_2}) = \al_{i_1 i_2} f(e_{i_2 i_2})
- f(e_{i_1 i_1}) \al_{i_1 i_2}$.  

Let $B{\cal{C}}$ denote the classifying space of the category
$\cal{C}$ \cite{Quillen}.  Thus, $B {\cal{C}}$ is the geometric
realization of the simplicial set $B_* ({\cal{C}})$, also called the
nerve of $\cal{C}$, where
$B_0 ({\cal{C}}) = {\rm{Obj}} ({\cal{C}})$, $B_1 ({\cal{C}}) =
{\rm{Mor}} ({\cal{C}})$.  For $\al_{ij} \in {\rm{Mor}}(i, \, j)$, $n
>1$, we have
$$  B_n ({\cal{C}}) = \{ (\al_{i_0 i_1}, \, \al_{i_1 i_2}, \, \ldots \,
, \, \al_{i_{n-1} i_n}) \ | \ i_0 \preccurlyeq i_1 \preccurlyeq \ldots 
\preccurlyeq i_n \}.  $$
The face maps $d_0$, $d_1 : B_1 ({\cal{C}}) \to B_0 ({\cal{C}})$ are
given by $d_0 (\al_{ij}) = j$ and $d_1 (\al_{ij}) = i$.  The
degeneracy $s_0 : B_0 ({\cal{C}}) \to B_1 ({\cal{C}})$ is simply $s_0
(i) = e_{ii}$.  The face maps $d_i : B_n ({\cal{C}}) \to
B_{n-1}({\cal{C}})$, $i = 0$, 1, 2, $\ldots \,$, $n$, and the
degeneracies $s_i : B_n ({\cal{C}}) \to B_{n+1} ({\cal{C}})$, 
$i = 0$, 1, 2, $\ldots \,$, $n$, are the same as those in the bar
construction \cite{Quillen}.  Let $B_n = B_n( {\cal{C}})$.  The
simplicial homology groups, $H_* (B{\cal{C}}; \, k)$, are computed
from the complex
$$  k[B_0] \overset{d}{\lla} k[B_1] \overset{d}{\lla} \, \ldots \,
\overset{d}{\lla} k[B_{n-1}] \overset{d}{\lla} k[B_n]
\overset{d}{\lla} \, \ldots \, ,  $$
where $d : k[B_n] \to k[B_{n-1}]$ is given by 
$d = \sum_{i=0}^n (-1)^i d_i$.  The simplicial cohomology groups, 
$H^*  (B{\cal{C}}; \, k)$, are computed from the ${\rm{Hom}}_k$-dual
of the $k[B_*]$-complex.  Let $B^n = {\rm{Hom}}_k (k[B_n], \, k)$.
Define 
$$  \Psi : B^0 \to C^0 = \sum_{i=1}^N {\rm{Hom}}_k(k_i, \, k[G_i])  $$
by $\Psi (\gamma) (e_{ii}) = \gamma (i) \, e_{ii}$.  For $n \geq 1$, 
define $\Psi : B^n \to AP(k[{\cal{C}}]^{\ot_E n}, \, k[{\cal{C}}])$ by
\begin{align*}
&  \Psi(\gamma)( \al_{i_1 i_2} \ot \al_{i_2 i_3} \ot \ldots \ot
\al_{i_{n-1} i_n}) \\
& = \gamma ( \al_{i_1 i_2} \ot \al_{i_2 i_3} \ldots \ot
\al_{i_{n-1} i_n})  
( \al_{i_1 i_2} \, \al_{i_2 i_3} \ldots \al_{i_{n-1} i_n}), 
\end{align*}
where $i_1 \preccurlyeq i_2 \preccurlyeq \ldots \preccurlyeq i_n$.  
Then $\Psi : B^* \to C^*$ is a cochain map with image being the
autopoietic cochains, but not always strictly so.

\begin{lemma}
There is a $k$-module isomorphism
$$  AP( k[{\cal{C}}]^{\ot n}, \, k[{\cal{C}}]) \simeq  
 AP( k[{\cal{C}}]^{\ot_E n}, \, k[{\cal{C}}]) .  $$
\end{lemma}
\begin{proof}
The proof follows essentially from Lemma \eqref{unique-morphism}.
Also, for morphisms $\al_j$ of $\cal{C}$ and
$$  \al_1 \ot \ldots \al_i \ot \al_{i+1} \ot \ldots \ot \al_n \in
k[{\cal{C}}]^{\ot n},  $$
if $\al_i$ and $\al_{i+1}$ are not composable, then
$\al_1 \ot \ldots \al_i \ot \al_{i+1} \ot \ldots \ot \al_n = 0$ in
$k[{\cal{C}}]^{\ot_E n}$.  For 
$f \in AP( k[{\cal{C}}]^{\ot n}, \, k[{\cal{C}}])$, we have
\begin{align*}
& f( \al_1 \ot \ldots \al_i \ot \al_{i+1} \ot \ldots \ot \al_n ) 
= \lambda (\al_1 \ldots \al_i \al_{i+1} \ldots \al_n) = 0.
\end{align*}
\end{proof} 
For $n \geq 1$, the non-autopoietic cochains in
${\rm{Hom}}_k(k[{\cal{C}}]^{\ot n} , \, k[{\cal{C}}])$ can be identified with 
$\sum_{i=1}^N NP( k[G_i]^{\ot n}, \, k[G_i])$ studied in \S 2.  Of course, for
$* \geq 0$, $NP( k[G_i]^{\ot *}, \, k[G_i])$ forms a subcomplex of 
${\rm{Hom}}_k ( k[G_i]^{\ot *}, \, k[G_i])$. 

\begin{theorem}
Let $k$ be a commutative ring with unit and let $\cal{C}$ be a
category formed by an amalgam of groups and a poset as defined above.
For the category algebra $k[{\cal{C}}]$, there is an isomorphism of
$k$-modules 
$$  HH^* (k[{\cal{C}}]; \, k[{\cal{C}}]) \simeq
H^*( B{\cal{C}}) \oplus \Big( \sum_{i=1}^N HH^*(k[G_i]; \,
k[G_i])/H^*(BG_i) \Big),  $$
where $H^*(BG_i)$ is identified with $H^*(AP(k[G_i]^{\ot *}))$.  Note that
$H^*(B {\cal{C}})$ is a subalgebra of $HH^* (k[{\cal{C}}]; \,
k[{\cal{C}}])$, realized by the $E_{\infty}$-subalgebra of autopoietic
Hochschild cochains.  
\end{theorem} 
\begin{proof}
The image $\Psi [B^0] = \sum_{i=1}^N AP(k, \, k[G_i])$.  The
cokernel $C^0 / \Psi[B^0]$ can be identified with the non-autopoietic
cochains $\sum_{k=1}^n NP(k, \, k[G_i])$.  We have
$$  \Psi [B^n] = AP(k[{\cal{C}}]^{\ot_E n}, \, k[{\cal{C}}]),
\ \ \ n \geq 1 .  $$
If $f : k[{\cal{C}}]^{\ot_E n} \to  k[{\cal{C}}]$ is not autopoietic,
then $f \in \sum_{i=1}^N NP(k[G_i]^{\ot n}, \, k[G_i])$.  Now, 
$\sum_{i=1}^N NP(k[G_i]^{\ot *}, \, k[G_i])$, $* \geq 0$, is a
subcomplex of $\{ C^* \}$ defined in \eqref{C-*-complex}.  The terms 
$\sum_{i=1}^N HH^*(k[G_i]; \, k[G_i])/H^*(BG_i)$ then follow from
Lemma \eqref{direct-sum}.   
\end{proof}
\begin{corollary}
Let $X$ be the wedge 
$X = B{\cal{C}} \vee \big( \vee_{i=1}^N {\rm{Maps}}(S^1, \, BG_i)/BG_i \Big)$,
where $BG_i$ is the subspace of constant loops in 
${\rm{Maps}}(S^1, \, BG_i)$.  Then
$$    HH^* (k[{\cal{C}}]; \, k[{\cal{C}}]) \simeq H^*(X; \, k).  $$
The above $k$-module isomorphism can be used to induce a simplicial
cup product on $HH^* (k[{\cal{C}}]; \, k[{\cal{C}}])$ that agrees with
the Gerstenhaber product on autopoietic cochains.  
\end{corollary}

\bigskip
\noindent
The corresponding author states that there is no conflict of interest.


\begin{thebibliography}{99}

\bibitem{Benson} Benson, D., {\em Representations and Cohomology II:
    Cohomology of Groups and Modules}, Cambridge University Press,
  Cambridge, 1991.

\bibitem{Cartan-Eilenberg} Cartan, H., Eilenberg, S., {\em Homological
    Algebra}, Princeton University Press, Princeton, New Jersey, 1956. 

\bibitem{Costello} Costello, K., ``Topological Conformal Field Theories
  and Calabi-Yau Categories,'' {\em Advances in Mathematics}, 
{\bf{210}}, 1, (2007) 165--214.

\bibitem{Gerstenhaber1962}  Gerstenhaber, M.,  ``The Cohomology Structure of
  an Associative Algebra,'' {\em Annals of Mathematics}, 78 (2), (1962)
  267--288. 

\bibitem{Gerstenhaber1983}  Gerstenhaber, M., Schack, S.,  ``Simplicial
  Cohomology is Hochschild Cohomology,'' {\em Journal of Pure and
    Applied Algebra}, {\bf{30}}, (1983) 143--156.

\bibitem{Gerstenhaber2014}  Gerstenhaber, M.,  ``Self-dual and Quasi
    Self-dual Algebras,'' {\em Israel Journal of Mathematics},
    {\bf{200}}, (2014) 193--211.

\bibitem{Gerstenhaber2017} Gerstenhaber, M., ``Path Algebras,
  Wave-Particle Duality, and Quantization of Phase Space,'' {\em
    Letters in Mathematical Physics}, {\bf{107}}, 3, (2017) 409--426.

\bibitem{Kontsevich1999} Kontsevich, M., ``Operads and Motives in
  Deformation Quantization,'' {\em Letters in Mathematical Physics},
  {\bf{48}}, (1999) 35--72.

\bibitem{Kontsevich2000} Kontsevich, M., Soibelman, Y.,
  ``Deformations of Algebras over Operads and the Deligne
  Conjecture,'' in {\em Conf\'erence Mosh\'e Flato}, Vol. I,
  Mathematical Physics Studies, 22,  Kluwer, Dordrecht, 2000.


\bibitem{Kontsevich2008} Kontsevich, M., Soibelman, Y., ``Notes on
  $A_{\infty}$-Algebras, $A_{\infty}$-Categories and Non-Commutative
  Geometry,'' in {\em Homological Mirror Symmetry}, Lecture Notes in Physics,
  757, Springer Verlag, Heidelberg, 2008.

\bibitem{Loday} Loday, J.-L., {\em Cyclic Homology}, second ed.,
  Grundlehren der Mathematischen Wissenschaften, Vol. 301, Springer
  Verlag, Berlin, 1998.

\bibitem{Lodder1} Lodder, J., ``A Comparison of Products in Hochschild
  Cohomology,'' {\em Communications in Algebra}, {\bf{44}}, 11, (2016)
  4874--4891.

\bibitem{Lodder2} Lodder, J.,  ``Hochschild Cohomology of Poset
  Algebras and Steenrod Operations,'' {\em Comptes Rendus de la 
Acad\'emie des Sciences}, {\bf{354}}, 4, (2016) 339--343.

\bibitem{McClure-Smith} McClure, J., Smith, J., ``Multivariable
  Cochain Operations and Little $n$-Cubes,'' {\em Journal of the
    American Mathematical Society}, {\bf{16}}, 3, (2003) 681--704.

\bibitem{Mosher} Mosher, R., Tangora, M., {\em Cohomology Operations
    and Applications in Homotopy Theory}, Harper and Row, New York,
  1968, reprinted by Dover Publications, New York, 2008.

\bibitem{Quillen} Quillen, D., ``Higher Algebraic K-theory I,''
  Lecture Notes in Math., Vol. 341, Springer Verlag, Berlin, 1973.

\bibitem{Steenrod} Steenrod, N., ``Products of Cocycles and Extensions
  of Mappings,'' {\em Annals of Mathematics}, {\bf{ 48}}, 2, (1947)
  290--320.  

\bibitem{Tamarkin} Tamarkin, D., ``Formality of Chain Operad of Little
 Discs,'' {\em Letters in Mathematical Physics}, {\bf{66}}, 1, 2,
 (2003) 65--72.

\bibitem{Voronov} Voronov, A., ``Homotopy Gerstenhaber Algebras,'' 
in {\em Conf\'erence Mosh\'e Flato}, Vol. II,
  Mathematical Physics Studies, 22,  Kluwer, Dordrecht, 2000.

\bibitem{Xu} Xu, F., ``Hochschild and Ordinary Cohomology Rings of
  Small Categories,'' {\em Advances in Mathematics}, {\bf{219}} (2008)
1872--1893.  


\end{thebibliography}
\end{document}